\documentclass[11pt]{article}
\usepackage{amsfonts}
\usepackage{amsmath,color}
\usepackage[makeroom]{cancel}
\usepackage{bbm}
\usepackage[affil-it]{authblk}

\newcommand{\beq}{\begin{equation}}
\newcommand{\eeq}{\end{equation}}
\newcommand{\bea}{\begin{eqnarray}}
\newcommand{\eea}{\end{eqnarray}}
\newcommand{\beas}{\begin{eqnarray*}}
\newcommand{\eeas}{\end{eqnarray*}}

\newtheorem{theorem}{Theorem}[section]

\newtheorem{definition}[theorem]{Definition}

\newtheorem{corollary}[theorem]{Corollary}

\newtheorem{remark}[theorem]{Remark}
\newtheorem{example}[theorem]{Example}
\newtheorem{examples}[theorem]{Examples}
\newtheorem{foo}[theorem]{Remarks}

\newenvironment{proof}{\addvspace{\medskipamount}\par\noindent{\it
Proof}.}
{\unskip\nobreak\hfill$\Box$\par\addvspace{\medskipamount}}

\DeclareMathOperator\arctanh{arctanh}
\DeclareMathOperator\arcosh{arcosh}

\newcommand{\p}{\mathbb P}

\newcommand{\M}{\mathbb M}
\newcommand{\bH}{\mathbb H}
\newcommand{\bS}{\mathbb S}
\newcommand{\CH}{\mathbb{CH}^n}

\newcommand{\R}{\mathbb R}

\newcommand{\E}{\mathbb E}
\title{Stochastic areas, Winding numbers and Hopf fibrations}

\author{Fabrice Baudoin%
  \thanks{Author supported in part by Grant NSF-DMS 15-11-328}}
\affil{Department of Mathematics, University of Connecticut}

\author{Jing Wang}
\affil{Department of Mathematics, University of Illinois}

\date{}

\begin{document}
\maketitle

\begin{abstract}
We define and study stochastic areas processes associated with  Brownian motions on the complex symmetric spaces $\mathbb{CP}^n$ and $\mathbb{CH}^n$. The characteristic functions of those processes are computed  and  limit theorems are obtained. In the case $n=1$, we also study windings of the Brownian motion on those spaces and compute the limit distributions. For $\mathbb{CP}^n$ the geometry of the Hopf fibration plays a central role, whereas for $\mathbb{CH}^n$ it is the anti-de Sitter fibration.
\end{abstract}

\begin{center}
\textit{Dedicated to Marc Yor}
\end{center}

\tableofcontents

\section{Introduction}

By its simplicity, its number of far reaching applications, and its connections to many areas of mathematics,  the L\'evy's area formula is undoubtedly among the most important and beautiful formulas in stochastic calculus. In this paper, we describe analogues and consequences of the formula on  the complex symmetric spaces $\mathbb{CP}^n$ and $\mathbb{CH}^n$.

\

\textbf{The L\'evy's area formula}

\

Let $Z_t=X_t+iY_t$, $t\ge0 $, be a  Brownian motion in the complex plane such that $Z_0=0$. The algebraic area swept out by the path of $Z$ up to time $t$ is given by
\begin{equation*}
S_t=\int_{Z[0,t]} xdy-yx=\int_0^t X_s dY_s -Y_sdX_s,
\end{equation*}
where the stochastic integral is an It\^o integral, or equivalently a Stratonovitch integral since the quadratic covariation between $X$ and $Y$ is 0. The L\'evy's area formula
\begin{equation}\label{LA}
\mathbb{E}\left( e^{i\lambda S_t} | Z_t=z\right)=\frac{\lambda t}{\sinh \lambda t} e^{-\frac{|z|^2}{2t}(\lambda t \coth \lambda t -1) }
\end{equation}
was originally proved in \cite{Levy} by using a series expansion of $Z$. The formula nowadays admits many different proofs. A particularly elegant probabilistic approach is due to Yor \cite{Yor} (see also \cite{williams}). The first observation is that, due to the invariance by rotations of $Z$, one has for every $\lambda \in \R$,
\[
\mathbb{E}\left( e^{i\lambda S_t} | Z_t=z\right)=\mathbb{E}\left( \left. e^{-\frac{\lambda^2}{2} \int_0^t |Z_s|^2 ds} \right| |Z_t|=|z|  \right).
\] 
One considers then the new probability
\[
\mathbb{P}_{/ \mathcal{F}_t}^\lambda = \exp \left( \frac{\lambda}{2}(|Z_t|^2 -2t) -\frac{\lambda^2}{2} \int_0^t |Z_s|^2 ds \right)\mathbb{P}_{/ \mathcal{F}_t}
\]
under which, thanks to Girsanov theorem, $(Z_t)_{t \ge 0}$ is a Gaussian process (an Ornstein-Uhlenbeck process). Formula \eqref{LA} then easily follows from standard computations on Gaussian measures.

\

Somewhat surprisingly, formula \eqref{LA} and the stochastic area process $(S_t)_{t \ge 0}$ appear in many different contexts. For instance, formula \eqref{LA} has  been used by Bismut \cite{Bi1,Bi2} in a probabilistic approach to index theory and allows to construct explicit parametrices for the heat equation on vector bundles (see \cite{B1}). Also, the Mellin transform of $S_t$ is closely related to the Riemann zeta function (see \cite{Biane}). We also point out that the stochastic area process is a central piece in the rough paths theory construction (see \cite{FV}). 

\

The stochastic area process is also intimately connected to sub-Riemannian geometry. More precisely, in his paper \cite{Gaveau}, Gaveau actually observes that the $2n+1$-dimensional process $(Z^1_t,\dots, Z^n_t, S_t)$ is a horizontal Brownian motion on the $2n+1$-dimensional isotropic Heisenberg group, where $Z^j_t=X^j_t+iY^j_t$ are independent Brownian motions on the complex plane and $S_t=\sum_{j=1}^n\int_0^t X^j_s dY^j_s -Y^j_sdX^j_s$. As a consequence of 3-dimensional example, formula \eqref{LA} yields an expression for the heat kernel of the  sub-Laplacian on the Heisenberg group.

\

\textbf{Stochastic area in $\mathbb{CP}^n$ }

\

In the first part of the paper, for a Brownian motion $(W_t)_{t\ge 0}$ on the projective space  $\mathbb{CP}^n$, we introduce and study a natural functional $\theta$ of $W$ that plays the role of a generalized stochastic area. It can be written as
\[
\theta_t=\int_{W[0,t]} \alpha
\]
where $\alpha$ is a one-form on $\mathbb{CP}^n$ such that almost everywhere $d\alpha$ is the K\"ahler form of $\mathbb{CP}^n$. This functional $\theta$ enjoys as many symmetries as the stochastic area  process in the plane. In particular, paralleling the Gaveau's geometric interpretation of the stochastic area in the plane, we prove (see Theorem \ref{horizon}) that $\theta$ can actually be interpreted as the fiber motion of the horizontal Brownian motion of the Hopf fibration
\[
\mathbf{U}(1) \to \bS^{2n+1} \to \mathbb{CP}^n.
\]
This geometric interpretation allows us to work in a system of coordinates that displays enough symmetries to carry out Yor's method. As a byproduct (see Theorem \ref{FThj}), we obtain an analogue for $\mathbb{CP}^n$ of the L\'evy's area formula \eqref{LA}. As a nice consequence, we prove (see Theorem \ref{limit CP})  that when $t \to +\infty$, in distribution
\[
\frac{\theta_t}{t} \to \mathcal{C}_n,
\]
where $\mathcal{C}_n$ is a Cauchy distribution with parameter $n$. \

\

\textbf{Stochastic area in $\mathbb{CH}^n$ }

\

In the second part of the paper, we carry out a similar program for the complex hyperbolic space $\mathbb{CH}^n$. For a Brownian motion $(W_t)_{t \ge 0}$ on $\mathbb{CH}^n$, we introduce a functional
\[
\theta_t=\int_{W[0,t]} \alpha
\]
where $\alpha$ is a one-form on $\mathbb{CH}^n$ such that  $d\alpha$ is the K\"ahler form of $\mathbb{CH}^n$. The process $\theta$ is then interpreted (see Theorem \ref{FGHT}) as the fiber motion of the horizontal Brownian motion of the anti-de Sitter fibration
\[
\mathbf{U}(1) \to \bH^{2n+1} \to \mathbb{CH}^n.
\]
As before, this interpretation allows us to study the law of $\theta_t$ and to compute its characteristic function. We can get a little further than in the case of $\mathbb{CP}^n$. In particular, when $n=1$, we get an explicit form for the density of the stochastic area of the Brownian loops in $\mathbb{CH}^1$  (see Theorem \ref{HJKLP}):  For $ t>0$, and $\theta \in \R$,
\[
\mathbb{P} \left( \theta_t \in d\theta | W_t=0 \right)=\frac{1}{C(t)} \frac{e^{-\frac{\theta^2}{2t}}}{\cosh^2 \left( \frac{\pi \theta}{t}\right)} d\theta.
\]
This formula is surprisingly close to the corresponding formula for the stochastic area of the Brownian loop in the plane (see \cite{D})
\[
\mathbb{P} \left( S_t \in ds | Z_t=0 \right)=\frac{1}{C(t)} \frac{1}{\cosh^2 \left( \frac{\pi s}{t}\right)} ds,
\]
and it would be interesting to find the connection between the two.  For any $n\ge 1$, we prove then (see Theorem \ref{LimitCHn}) that when $t \to +\infty$, the following convergence in distribution takes place
\[
\frac{\theta_t}{\sqrt{t}} \to \mathcal{N}(0,1)
\]
where $\mathcal{N}(0,1)$ is a normal distribution with mean 0 and variance 1. It is worth pointing out that the limit distribution is therefore independent of the complex dimension $n$. 

\

\textbf{Winding numbers}

\

In the case $n=1$, stochastic area processes are closely related to  winding numbers  $\phi_t$ of the Brownian motion around a fixed point.  In the case of $\mathbb{CP}^1$, by adapting Yor's method \cite{Yor1},  we recover the fact  that when $t \to +\infty$, in distribution we have a Spitzer's type behavior 
\[
\frac{\phi_t}{t} \to \mathcal{C}_2,
\]
where $\mathcal{C}_2$ is a Cauchy distribution with parameter 2.  This is not surprising since Spitzer's law is understood on general compact Riemann surfaces (see \cite{watanabe}). In the case of $\mathbb{CH}^1$  we prove that when $t \to \infty$, in distribution we have
\[
{\phi_t}\to \mathcal{C},
\]
where $\mathcal{C}$ is a Cauchy distribution with parameter $-\ln |z|$, where $z$ is the starting point of the Brownian motion. The finiteness of  $\phi_t$ when $t \to \infty$ is obviously not surprising since the Brownian motion on $\mathbb{CH}^1$ is transient.

\section{Generalized stochastic areas in $\mathbb{CP}^n$}

\subsection{Stochastic area and Hopf fibration}

The complex projective space $\mathbb{CP}^n$ can be defined as the set of complex lines in $\mathbb{C}^{n+1}$. To parametrize points in $\mathbb{CP}^n$, it is convenient to  use the local inhomogeneous coordinates given by $w_j=z_j/z_{n+1}$, $1 \le j \le n$, $z \in \mathbb{C}^{n+1}$, $z_{n+1}\neq 0$. In these coordinates, the Riemannian structure of $\mathbb{CP}^n$ is easily worked out from the standard Riemannian structure of the Euclidean sphere. Indeed, if we consider the unit sphere
\[
\bS^{2n+1}=\lbrace z=(z_1,\cdots,z_{n+1})\in \mathbb{C}^{n+1}, \| z \| =1\rbrace,
\]
then, at each point,  the differential of the map $\bS^{2n+1} -\{z_{n+1}=0 \}  \to \mathbb{CP}^n$, $ (z_1,\cdots,z_{n+1}) \to (z_1/z_{n+1},\cdots,z_n/z_{n+1})$ is an isometry between the orthogonal space of its kernel and the corresponding tangent space to $\mathbb{CP}^n$. This map actually is the local description of a globally defined Riemannian submersion $\pi: \bS^{2n+1} \to \mathbb{CP}^n$, that can be constructed as follows. There is an isometric group action of $\mathbb{S}^1=\mathbf{U}(1)$ on $\bS^{2n+1}$ which is  defined by $$e^{i\theta}\cdot(z_1,\cdots, z_n) = (e^{i\theta} z_1,\cdots, e^{i\theta} z_n). $$

The quotient space $\bS^{2n+1} / \mathbf{U}(1)$ can be identified with $\mathbb{CP}^n$ and the projection map $$\pi :  \bS^{2n+1} \to \mathbb{CP}^n$$ is a Riemannian submersion with totally geodesic fibers isometric to $\mathbf{U}(1)$. The fibration
\[
\mathbf{U}(1) \to \bS^{2n+1} \to \mathbb{CP}^n
\]
 is called the Hopf fibration. 

\

The submersion $\pi$ allows one to construct the Brownian motion on $\mathbb{CP}^n$ from the Brownian motion on $\bS^{2n+1}$. Indeed,  let $(z(t))_{t \ge 0}$ be a Brownian motion on  $\bS^{2n+1}$ started at the north pole \footnote{We will call north pole the point with complex coordinates $z_1=0,\cdots, z_{n+1}=1$. }. Since $\mathbb{P}( \exists t \ge 0, z_{n+1}(t)=0 )=0$, one can use the local description of the submersion $\pi$ to infer that
\begin{align}\label{BMsphere}
w(t)= \left( \frac{z_1(t)}{z_{n+1}(t)} , \cdots, \frac{z_n(t)}{z_{n+1}(t)}\right), \quad t \ge 0,
\end{align}
is a Brownian motion on $\mathbb{CP}^n$.

\

Consider the one-form $\alpha$ on $\mathbb{CP}^n$ given by the push-forward: $\alpha \left( \frac{\partial}{\partial w_i} \right)=\eta  \left( \frac{\partial}{\partial w_i} \right)$, where $\eta$ is the standard contact form on $\bS^{2n+1}$. In local inhomogeneous coordinates,  one has
\[
\alpha=\frac{i}{2(1+|w|^2)}\sum_{j=1}^n(w_jd\overline{w_j}-\overline{w_j}dw_j)
\]
where $|w|^2=\sum_{j=1}^n |w_j|^2$. It is easy to compute that
\begin{align*}
d\alpha=\frac{i}{(1+|w|^2)^2}\left((1+|w|^2)\sum_{j=1}^ndw_j\wedge d\overline{w}_j-\sum_{j,k=1}^n\overline{w}_jw_k dw_j\wedge d\overline{w}_k \right).
\end{align*}
Thus $d\alpha$ is almost everywhere the K\"ahler form   that induces the standard Fubini-Study metric on $\mathbb{CP}^n$. Observe that $\alpha$ is only defined in the  inhomogeneous coordinates chart and may not be smoothly extended to the whole $\mathbb{CP}^n$, because on compact K\"ahler manifolds the K\"ahler form can not be exact.

\begin{definition}
Let $(w(t))_{t \ge 0}$ be a Brownian motion on $\mathbb{CP}^n$ started at $0$\footnote{We call $0$ the point with inhomogeneous coordinates $w_1=0,\cdots, w_{n}=0$}. The generalized stochastic area process of $(w(t))_{t \ge 0}$ is defined by
\[
\theta(t)=\int_{w[0,t]} \alpha=\frac{i}{2}\sum_{j=1}^n \int_0^t \frac{w_j(s)  d\overline{w_j}(s)-\overline{w_j}(s) dw_j(s)}{1+|w(s)|^2},
\]
where the above stochastic integrals are understood in the Stratonovitch, or equivalently in the It\^o sense.
\end{definition}

Our goal is to study the distribution of $\theta(t)$. In particular we aim to prove (see Theorem \ref{limit CP}) that in distribution
\[
\frac{\theta(t)}{t} \to \mathcal{C}_n,
\]
where $\mathcal{C}_n$ is a Cauchy distribution with parameter $n$

\

The first step is to show that the stochastic area process is closely related to a diffusion process on $\bS^{2n+1}$.

 \begin{theorem}\label{horizon}
 Let $(w(t))_{t \ge 0}$ be a Brownian motion on $\mathbb{CP}^n$ started at 0 and $(\theta(t))_{t\ge 0}$ be its stochastic area process. The $\mathbb{S}^{2n+1}$-valued diffusion process
 \[
 X_t=\frac{e^{-i\theta(t)} }{\sqrt{1+|w(t)|^2}} \left( w(t),1\right), \quad t \ge 0
 \]
 is the horizontal lift at the north pole of $(w(t))_{t \ge 0}$ by the submersion $\pi$.
 \end{theorem}
 
 \begin{proof}
 The key point is that the submersion $\pi$ is compatible with the contact structure of $\bS^{2n+1}$. More precisely, the horizontal distribution of this submersion is the kernel of the standard contact form on $\bS^{2n+1}$ 
\[
\eta=-\frac{i}{2}\sum_{j=1}^{n+1}(\overline{z_j}dz_j-z_jd\overline{z_j}),
\]
and the fibers of the submersion are the orbits of the Reeb vector field.
\

As above, let $(w_1,\cdots, w_n)$ be  the local inhomogeneous coordinates for $\mathbb{CP}^n$ given by $w_j=z_j/z_{n+1}$, and $\theta$ be the local fiber coordinate, i.e., $(w_1, \cdots, w_n)$ parametrizes the complex lines passing through the north pole, while $\theta$ determines a point on the line that is of unit distance from the north pole. These coordinates are given by the map
\begin{align}\label{invar}
(w,\theta)\longrightarrow \frac{e^{i\theta} }{\sqrt{1+|w|^2}} \left( w,1\right),
\end{align}
where  $\theta \in \R/2\pi\mathbb{Z}$, and $w \in \mathbb{CP}^n$. In these coordinates, we compute that
 \begin{align*}
\eta =  d\theta+\frac{i}{2(1+|w|^2)}\sum_{k=1}^n(w_kd\overline{w_k}-\overline{w_k}dw_k).
\end{align*}
As a consequence, the horizontal lift to $\bS^{2n+1}$  of the vector field $\frac{\partial}{\partial w_i}$ is given by $\frac{\partial}{\partial w_i}-\alpha\left(\frac{\partial}{\partial w_i} \right)\frac{\partial}{\partial \theta}$. If we consider now a smooth curve $\gamma$ starting at 0 in $\mathbb{CP}^n$, we can write in the inhomogeneous system of coordinates
\[
\gamma(t)=(\gamma_1(t),\cdots,\gamma_n(t)).
\]
Since
\[
\gamma'(t)=\sum_{i=1}^n \gamma'_i(t)\frac{\partial}{\partial w_i},
\]
we deduce that the horizontal lift, $\bar{\gamma}$, of $\gamma$ at the north pole is given in the set of coordinates \eqref{invar}, by
\begin{align*}
\bar{\gamma}'(t)&=\sum_{i=1}^n \gamma'_i(t)\left(\frac{\partial}{\partial w_i}-\alpha\left(\frac{\partial}{\partial w_i} \right)\frac{\partial}{\partial \theta} \right) \\
 &=\sum_{i=1}^n \gamma'_i(t)\frac{\partial}{\partial w_i}- \sum_{i=1}^n \alpha\left(\frac{\partial}{\partial w_i} \right) \gamma'_i(t)\frac{\partial}{\partial \theta}.
\end{align*}
As a consequence,
\[
\bar{\gamma}(t)=\frac{e^{-i\theta(t)} }{\sqrt{1+|\gamma(t)|^2}} \left( \gamma(t),1\right),
\]
with
\[
\theta(t) =\sum_{i=1}^n \int_0^t \alpha\left(\frac{\partial}{\partial w_i} \right)  \gamma_i'(t)dt =\int_{\gamma[0,t]} \alpha.
\]

Similarly, the lift of the Brownian motion $(w(t))_{t \ge 0}$ is $\frac{e^{-i\theta(t)} }{\sqrt{1+|w(t)|^2}} \left( w(t),1\right)$ with
\[
\theta(t) =\sum_{i=1}^n \int_0^t \alpha\left(\frac{\partial}{\partial w_i} \right)  dw_i =\int_{w[0,t]} \alpha.
\]
\end{proof}

\begin{remark}
The diffusion process $(X_t)_{t \ge 0}$ is therefore the  horizontal Brownian motion on $\mathbb{S}^{2n+1}$. It is a hypoelliptic diffusion whose heat kernel was explicitly computed for $n=1$ in \cite{BB1} and for $n \ge 2$ in \cite{BW1}.
\end{remark}

The following theorem is the key to study the distribution of the process $(\theta(t))_{t \ge 0}$.

\begin{theorem}\label{diff1}
Let $r(t)=\arctan |w(t)|$. The process $\left( r(t), \theta(t)\right)_{t \ge 0}$ is a diffusion with generator
 \[
L=\frac{1}{2} \left( \frac{\partial^2}{\partial r^2}+((2n-1)\cot r-\tan r)\frac{\partial}{\partial r}+\tan^2r\frac{\partial^2}{\partial \theta^2}\right).
 \]
 As a consequence the following equality in distribution holds
 \[
\left( r(t) ,\theta(t) \right)_{t \ge 0}=\left( r(t),B_{\int_0^t \tan^2 r(s)ds}\right)_{t \ge 0},
\]
where $(B_t)_{t \ge 0}$ is a standard Brownian motion independent from $r$.
\end{theorem}

\begin{proof}
We first compute, in the coordinates \eqref{invar}, the generator of the diffusion $X$ introduced in the previous theorem. It is known that the Laplace-Beltrami operator for the Fubini-Study metric on $\mathbb{CP}^n$ is
\begin{equation}\label{eq-Laplacian-CPn}
\Delta_{\mathbb{CP}^n}=4(1+|w|^2)\sum_{k=1}^n \frac{\partial^2}{\partial w_k \partial\overline{w_k}}+ 4(1+|w|^2)\mathcal{R} \overline{\mathcal{R}}
\end{equation}
where
\[
\mathcal{R}=\sum_{j=1}^n w_j \frac{\partial}{\partial w_j}.
\] 
Since $X$ is the horizontal lift of $w$, the generator of $X$ is $\frac{1}{2}\bar{\Delta}_{\mathbb{CP}^n}$ where $\bar{\Delta}_{\mathbb{CP}^n}$ is the horizontal lift to  $\bS^{2n+1}$ of $\Delta_{\mathbb{CP}^n}$. As we have seen,  the horizontal lift to $\bS^{2n+1}$  of the vector field $\frac{\partial}{\partial w_i}$ is given by 
\[
\frac{\partial}{\partial w_i}-\alpha\left(\frac{\partial}{\partial w_i} \right)\frac{\partial}{\partial \theta}=\frac{\partial}{\partial w_j}+\frac{i}{2}\frac{\overline{w_j}}{1+\rho^2}\frac{\partial}{\partial\theta},
\]
where $\rho=|w|$. Substituting $\frac{\partial}{\partial w_i}$ by its lift in the expression of $\Delta_{\mathbb{CP}^n}$ yields
\begin{equation}\label{genera}
\bar{\Delta}_{\mathbb{CP}^n}=4(1+|w|^2)\sum_{k=1}^n \frac{\partial^2}{\partial w_k \partial\overline{w_k}}+ 4(1+|w|^2)\mathcal{R} \overline{\mathcal{R}}+|w|^2\ \frac{\partial^2}{\partial \theta^2}-2i(|w|^2+1)(\mathcal{R} -\overline{\mathcal{R}})\frac{\partial}{\partial\theta}.
\end{equation}
We compute then that $\bar{\Delta}_{\mathbb{CP}^n}$ acts on functions depending only on $(\rho, \theta)$ as
\[
\left(1+\rho^2\right)^2\frac{\partial^2}{\partial \rho^2}+\left(\frac{(2n-1)(1+\rho^2)}{\rho}+(1+\rho^2)\rho\right)\frac{\partial}{\partial \rho}+\rho^2\frac{\partial^2 }{\partial \theta^2}.
\]
The change of variable $\rho =\tan r$ finishes the proof.
\end{proof}

\begin{remark}
As a consequence of the previous proposition,  $r(t)=\arctan |w(t)|$ is a Jacobi diffusion (see Appendix for more details) with generator
\[
\frac{1}{2} \left( \frac{\partial^2}{\partial r^2}+((2n-1)\cot r-\tan r)\frac{\partial}{\partial r}\right).
\]
On the other hand, if $(\beta(t))_{t \ge 0}$ is a Brownian motion in $\mathbb{C}^{n+1}$, then from a classical skew-product decomposition (see \cite{PR}),
\begin{align*}
\frac{\beta(t)}{| \beta(t) |}=  z\left( \int_0^t \frac{ds}{ | \beta(s)|^2}\right),
\end{align*}
where $z(t)$ is a Brownian motion on the sphere $\bS^{2n+1}$. We deduce therefore from \eqref{BMsphere} that
\begin{align}\label{plki}
\frac{ \sqrt{ \sum_{i=1}^n | \beta_i (t)|^2}}{|\beta_{n+1} (t)|  }=\tan r \left(  \int_0^t \frac{ds}{ | \beta(s)|^2} \right).
\end{align}
The process $\sqrt{ \sum_{i=1}^n | \beta_i (t)|^2}$ is a Bessel process with dimension $2n$ and $| \beta_{n+1} (t)|$ is a Bessel process with dimension 2. The equality \eqref{plki} is therefore a special case of the general skew-product representation of Jacobi processes that was discovered by Warren and Yor \cite{WY}.
\end{remark}

\begin{remark}
The operator
\[
\frac{1}{2} \left( \frac{\partial^2}{\partial r^2}+((2n-1)\cot r-\tan r)\frac{\partial}{\partial r}+\tan^2r\frac{\partial^2}{\partial \theta^2}\right)
\]
with periodic boundary conditions on the endpoints of the interval $\theta \in [0,2\pi]$ was studied in \cite{BB1,BW1}. We cannot use those results in our case since $L$ does not come  with the same boundary conditions (in our case $\theta \in \R$).
\end{remark}

\subsection{Characteristic function of the stochastic area and limit theorem}

In this section, we study the characteristic function of $\theta(t)$.
Let $\lambda \ge 0$, $r \in [0,\pi/2)$ and 
\[
I(\lambda,r)=\mathbb{E}\left(e^{i \lambda \theta(t)}\mid r(t)=r\right).
\]
From Theorem \ref{diff1}, we know that
\begin{align*}
I(\lambda,r)& =\mathbb{E}\left(e^{i \lambda B_{\int_0^t \tan^2 r(s)ds}}\mid r(t)=r\right) \\
 &=\mathbb{E}\left(e^{- \frac{\lambda^2}{2} \int_0^t \tan^2 r(s)ds}\mid r(t)=r\right) 
\end{align*}
and $r$ is a diffusion with Jacobi generator
\[
\mathcal{L}^{n-1,0}=\frac{1}{2} \left( \frac{\partial^2}{\partial r^2}+((2n-1)\cot r-\tan r)\frac{\partial}{\partial r}\right)
\]
started at $0$.  The Jacobi generator 
\[
\mathcal{L}^{\alpha,\beta}=\frac{1}{2} \frac{\partial^2}{\partial r^2}+\left(\left(\alpha+\frac{1}{2}\right)\cot r-\left(\beta+\frac{1}{2}\right) \tan r\right)\frac{\partial}{\partial r}, \quad \alpha,\beta >-1
\]
is studied in details in the Appendix to which we refer for further details. We denote by $q_t^{\alpha,\beta}(r_0,r)$ the transition density with respect to the Lebesgue measure of the diffusion with generator $\mathcal{L}^{\alpha,\beta}$.
\begin{theorem}\label{FThj}
For $\lambda \ge 0$, $r \in [0,\pi/2)$, and $t >0$ we have
\begin{equation}\label{eq-ft-cond}
\mathbb{E}\left(e^{i \lambda \theta(t)}\mid r(t)=r\right)=\mathbb{E}\left(e^{- \frac{\lambda^2}{2} \int_0^t \tan^2 r(s)ds}\mid r(t)=r\right) =\frac{e^{-n \lambda t}}{(\cos r)^\lambda} \frac{q_t^{n-1,\lambda}(0,r)}{q_t^{n-1,0}(0,r)}.
\end{equation}
\end{theorem}

\begin{proof}
We have
\[
dr(t)= \frac{1}{2} \left( (2n-1)\cot r(t)-\tan r(t) \right)dt+d\gamma(t),
\]
where $\gamma$ is a standard Brownian motion.
Consider the local martingale
\begin{align*}
D_t& =\exp \left( -\lambda \int_0^t \tan r(s) d\gamma(s) -\frac{\lambda^2}{2}  \int_0^t \tan^2 r(s) ds \right)  \\
 &=\exp \left( -\lambda \int_0^t \tan r(s) dr(s)+\frac{\lambda}{2}(2n-1)t -\frac{\lambda +\lambda^2}{2}  \int_0^t \tan^2 r(s) ds \right) 
\end{align*}
From It\^o's formula, we have
\begin{align*}
\ln \cos r(t) & =-\int_0^t \tan r(s) dr(s)-\frac{1}{2} \int_0^t \frac{ds}{\cos^2 r(s)} \\
 &=-\int_0^t \tan r(s) dr(s)-\frac{1}{2} \int_0^t \tan^2 r(s) ds-\frac{1}{2} t.
\end{align*}
As a consequence, we deduce that
\[
D_t =e^{n\lambda t} (\cos r(t))^\lambda e^{- \frac{\lambda^2}{2} \int_0^t \tan^2 r(s)ds}.
\]
This expression of $D$ implies that almost surely $D_t \le e^{n\lambda t}$ and thus $D$ is a martingale. Let us denote by $\mathcal{F}$ the natural filtration of $r$ and consider the probability measure $\mathbb{P}^\lambda$ defined by
\[
\mathbb{P}_{/ \mathcal{F}_t} ^\lambda=e^{n\lambda t} (\cos r(t))^\lambda e^{- \frac{\lambda^2}{2} \int_0^t \tan^2 r(s)ds} \mathbb{P}_{/ \mathcal{F}_t}.
\]
We have then for every bounded and Borel function $f$ on $[0,\pi /2]$,
\begin{align*}
\mathbb{E}\left(f(r(t))e^{- \frac{\lambda^2}{2} \int_0^t \tan^2 r(s)ds}\right)=e^{-n\lambda t} \mathbb{E}^\lambda \left( \frac{f(r(t))}{(\cos r(t))^\lambda} \right)
\end{align*}
From Girsanov theorem, the process
\[
\beta(t)=\gamma(t)+\lambda \int_0^t \tan r(s) ds
\]
is a Brownian motion under the probability $\mathbb{P}^\lambda$. Since
\[
dr(t)= \frac{1}{2} \left( (2n-1)\cot r(t)-(2\lambda +1)\tan r(t) \right)dt+d\beta(t),
\]
the proof is complete.
\end{proof}

We deduce an expression for the Fourier transform of $\theta(t)$.
\begin{corollary}\label{cor-theta}
For $\lambda \in \mathbb{R}$ and $t \ge 0$,
\[
\mathbb{E}\left(e^{i \lambda \theta(t)}\right)=e^{-n | \lambda | t}\int_0^{\pi /2}  \frac{q_t^{n-1,| \lambda |}(0,r)}{(\cos r)^{| \lambda|}} dr 
\]
\end{corollary}
\begin{proof}
This is a direct consequence of \eqref{eq-ft-cond}.
\end{proof}
We are now in position to prove the main result of the section.

\begin{theorem}\label{limit CP}
When $t \to +\infty$, the following convergence in distribution takes place
\[
\frac{\theta(t)}{t} \to \mathcal{C}_n,
\]
where $\mathcal{C}_n$ is a Cauchy distribution with parameter $n$.
\end{theorem}

\begin{proof}
We just need to show that for $\lambda >0$, $\lim_{t\to+\infty}\mathbb{E}\left(e^{i \lambda \theta(t)/t}\right)=e^{-n  \lambda  }$. From Corollary \ref{cor-theta} we have for every $t>0$,
\[
\mathbb{E}\left(e^{i \lambda \frac{\theta(t)}{t}}\right)=e^{-n  \lambda  }\int_0^{\pi /2}  \frac{q_t^{n-1,\frac{ \lambda }{t}}(0,r)}{(\cos r)^{ \frac{\lambda}{t}}} dr 
\]
Using the formula for $q_t^{n-1, \lambda }(0,r)$ which is given in the Appendix, we obtain that
\[
\lim_{t \to \infty}  \int_0^{\pi /2}  \frac{q_t^{n-1,\frac{ \lambda }{t}}(0,r)}{(\cos r)^{ \frac{\lambda}{t}}} dr=  \int_0^{\pi /2}  q_\infty^{n-1,0}(0,r) dr =1.
\]
\end{proof}

\subsection{Skew-product decompositions on the Berger sphere and homogenisation}

In this section, we study the connection between the  stochastic area process on $\mathbb{CP}^n$ and an interesting class of diffusions on $\bS^{2n+1}$. These diffusions naturally arise when one looks at the canonical variation of the standard metric in the direction of the fibers of the Hopf fibration (see \cite{BB}).  We compute the density of these diffusions and prove a homogenisation result when the fibers collapse. The homogenisation problem is inspired by discussions with X.M. Li  and some of her results \cite{Li1,Li2}. 

\

As we mentioned before, the quotient space $\bS^{2n+1} / \mathbf{U}(1)$ is the projective complex space $\mathbb{CP}^n$ and the projection map $\pi :  \bS^{2n+1} \to \mathbb{CP}^n$ is a Riemannian submersion with totally geodesic fibers isometric to $\mathbf{U}(1)$. The sub-bundle $\mathcal{V}$ of $\mathbf{T}\bS^{2n+1}$ formed by vectors tangent to the fibers of the submersion is referred  to as the set of \emph{vertical directions}. The sub-bundle $\mathcal{H}$ of $\mathbf{T}\bS^{2n+1}$  which is normal to $\mathcal{V}$ is referred to as the set of \emph{horizontal directions}.   The standard metric  $g$ of of $\bS^{2n+1}$ can be split as
\begin{equation*}
g=g_\mathcal{H} \oplus g_{\mathcal{V}},
\end{equation*}
where the sum is orthogonal. We introduce the one-parameter family of Riemannian metrics:
\begin{equation}\label{eq-metric-B}
g_{\lambda}=g_\mathcal{H} \oplus  \frac{1}{\lambda^2 }g_{\mathcal{V}}, \quad \lambda >0,
\end{equation}
The Riemannian manifold $(\bS^{2n+1}, g_{\lambda})$ is called the Berger\footnote{For Marcel Berger (1927-).} sphere with parameter $\lambda >0$. The case $\lambda=1$ corresponds to the standard metric on $\bS^{2n+1}$. When $\lambda \to 0$, in the Gromov-Hausdorff sense,  $(\bS^{2n+1}, g_{\lambda})$ converges to $\bS^{2n+1}$ endowed with the Carnot-Carath\'eodory metric. When $\lambda \to \infty$, $(\bS^{2n+1}, g_{\lambda})$ converges to $\mathbb{CP}^n$ endowed with its standard Fubini-Study metric.

The following skew-product decomposition  for the Brownian motion on the Berger sphere holds.

\begin{theorem}
Let $\lambda >0$. Let $(w(t))_{t \ge 0}$ be a Brownian motion on $\mathbb{CP}^n$ started at 0 and $(\theta(t))_{t\ge 0}$ be its stochastic area process. Let $(B(t))_{t \ge 0}$ be a real Brownian motion independent from $w$. The $\mathbb{S}^{2n+1}$-valued diffusion process
 \[
 \beta^\lambda(t)=\frac{e^{i\lambda B(t)-i\theta(t)} }{\sqrt{1+|w(t)|^2}} \left( w(t),1\right), \quad t \ge 0
 \]
 is a Brownian motion on the Berger sphere $(\bS^{2n+1}, g_{\lambda})$.
\end{theorem}

\begin{proof}
Let $\lambda >0$ and let $ (\beta^\lambda(t))_{t \ge 0}$ be a Brownian motion on the Berger sphere $(\bS^{2n+1}, g_{\lambda})$. We work in the coordinates \eqref{invar}. From \eqref{eq-metric-B} and \eqref{genera} the generator of $\beta^\lambda$ is
\[
2(1+|w|^2)\sum_{k=1}^n \frac{\partial^2}{\partial w_k \partial\overline{w_k}}+ 2(1+|w|^2)\mathcal{R} \overline{\mathcal{R}}+\frac{1}{2}(\lambda^2+|w|^2)\ \frac{\partial^2}{\partial \theta^2}-i(|w|^2+1)(\mathcal{R} -\overline{\mathcal{R}})\frac{\partial}{\partial\theta}.
\]
This generator can be written as 
\[
L_X+\frac{1}{2}\lambda^2 \frac{\partial^2}{\partial \theta^2},
\]
where $L_X$ is the generator the diffusion $(X_t)_{t \ge 0}$ considered in Theorem \ref{horizon}. Since $L_X$ and $\frac{\partial^2}{\partial \theta^2}$ commute, the result easily follows.
\end{proof}
This representation of the Brownian motion on the Berger sphere yields an interesting result.

\begin{corollary}
Let $\lambda >0$ and let $ (\beta^\lambda(t))_{t \ge 0}$ be a Brownian motion on the Berger sphere $(\bS^{2n+1}, g_{\lambda})$. Let $\eta$ be the standard contact form on $\bS^{2n+1}$. The process
\[
\gamma(t)=\frac{1}{\lambda} \int_{\beta^\lambda [0,t]} \eta,
\]
is a real Brownian motion.
\end{corollary}

\begin{proof}
In the coordinates \eqref{invar}, we have
 \begin{align*}
\eta =  d\theta+\frac{i}{2(1+|w|^2)}\sum_{k=1}^n(w_kd\overline{w_k}-\overline{w_k}dw_k).
\end{align*}
Since
\[
\beta^\lambda(t)=\frac{e^{i\lambda B(t)-i\theta(t)} }{\sqrt{1+|w(t)|^2}} \left( w(t),1\right)
\]
we deduce that
\begin{align*}
 \int_{\beta^\lambda [0,t]} \eta &= \int_{\beta^\lambda [0,t]} d\theta+ \int_{\beta^\lambda [0,t]} \frac{i}{2(1+|w|^2)}\sum_{k=1}^n(w_kd\overline{w_k}-\overline{w_k}dw_k)  \\
  &=\int_0^t (\lambda dB(s) -d\theta (s))  +\frac{i}{2}\sum_{j=1}^n \int_0^t \frac{w_j(s)  d\overline{w_j}(s)-\overline{w_j}(s) dw_j(s)}{1+|w(s)|^2}\\
 &=\lambda B(t)-\theta(t)+\frac{i}{2}\sum_{j=1}^n \int_0^t \frac{w_j(s)  d\overline{w_j}(s)-\overline{w_j}(s) dw_j(s)}{1+|w(s)|^2} \\
 &=\lambda B(t).
\end{align*}
\end{proof}

We now turn to the homogenisation result.

\begin{theorem}
Let $ (\beta^\lambda(t))_{t \ge 0, \lambda >0}$ be a family of Brownian motions on the Berger spheres $(\bS^{2n+1}, g_{\lambda})$ started at the north pole. Let $f:\bS^{2n+1}\to \R$ be a bounded and Borel function. For every $t>0$, one has
\[
\lim_{\lambda \to \infty} \mathbb{E}\left( f(\beta^\lambda(t))\right)=\frac{1}{2\pi} \int_0^{2\pi} \mathbb{E}\left[f \left( \frac{e^{i\theta} }{\sqrt{1+|w(t)|^2}} \left( w(t),1\right)\right)\right] d\theta,
\]
where $w$ is a Brownian motion on $\mathbb{CP}^n$ started at 0.
\end{theorem}

\begin{proof}
We already know that,  in the coordinates \eqref{invar},  the generator of $\beta^\lambda$ is
\[
2(1+|w|^2)\sum_{k=1}^n \frac{\partial^2}{\partial w_k \partial\overline{w_k}}+ 2(1+|w|^2)\mathcal{R} \overline{\mathcal{R}}+\frac{1}{2}(\lambda^2+|w|^2)\ \frac{\partial^2}{\partial \theta^2}-i(|w|^2+1)(\mathcal{R} -\overline{\mathcal{R}})\frac{\partial}{\partial\theta}.
\]
By symmetry, the heat kernel of $\beta^\lambda$ only depends of the variables $r=\arctan |w|$ and $\theta$. The radial part of the generator of $\beta^\lambda$ is then
\[
L_\lambda=\frac{1}{2}\left(\frac{\partial^2}{\partial r^2}+((2n-1)\cot r-\tan r)\frac{\partial}{\partial r}+(\tan^2r+\lambda^2)\frac{\partial^2}{\partial \theta^2}\right).
\]
The heat kernel $q_t^\lambda(r,\theta)$ of the diffusion with generator $L_\lambda$ can be computed as in \cite{BW1} and one obtains
\[
q_t^\lambda(r,\theta)=\frac{\Gamma(n)}{2\pi^{n+1}}\sum_{k=-\infty}^{+\infty}\sum_{m=0}^{+\infty} (2m+|k|+n){m+|k|+n-1\choose n-1}e^{-\frac{1}{2}(\lambda_{m,k}+k^2\lambda^2)t+ik \theta}(\cos r)^{|k|}P_m^{n-1,|k|}(\cos 2r),
\]
where $\lambda_{m,k}=4m(m+|k|+n)+2|k|n$ and $P_m^{n-1,|k|}$ are  Jacobi polynomials. Since 
\begin{align*}
 &\bigg|q_t^\lambda(r,\theta)-\frac{\Gamma(n)}{2\pi^{n+1}}\sum_{m=0}^{+\infty} (2m+n){m+n-1\choose n-1}e^{-\frac{1}{2}\lambda_{m,0}t}P_m^{n-1,|0|}(\cos 2r)\bigg|\\
&\le\frac{\Gamma(n)}{2\pi^{n+1}}\sum_{k\not=0}\sum_{m=0}^{+\infty} e^{-\frac{1}{2}k^2\lambda^2 t}\bigg|(2m+|k|+n){m+|k|+n-1\choose n-1}e^{-\frac{1}{2}\lambda_{m,k}t}(\cos r)^{|k|}P_m^{n-1,|k|}(\cos 2r)\bigg|,
\end{align*}
one easily deduces that in $L^2(\mu)$, for $t>0$,
\[
\lim_{\lambda \to +\infty} q_t^\lambda(r,\theta)=q_t^\infty(r),
\]
where
\[
q_t^\infty(r)=\frac{\Gamma(n)}{2\pi^{n+1}}\sum_{m=0}^{+\infty} (2m+n){m+n-1\choose n-1}e^{-2m(m+n)t}P_m^{n-1,0}(\cos 2r)
\]
and
\[
d\mu=\frac{2\pi^n}{\Gamma(n)}(\sin r)^{2n-1}\cos r drd\theta,
\]
is the invariant and symmetric measure of $L_\lambda$. We now conclude by observing that $q_t^\infty(r)$ is the heat kernel at 0 of the radial part of the Brownian motion on $\mathbb{CP}^n$.
\end{proof}

\section{Generalized stochastic areas in $\mathbb{CH}^n$}

We now turn to the second part of the paper.

\subsection{Stochastic area and anti-de Sitter fibration}

As a set, the  complex hyperbolic space $\mathbb{CH}^n$ can be defined as the open unit ball in $\mathbb{C}^n$. Its Riemannian structure can be constructed as follows. Let 
\[
\bH^{2n+1}=\{ z \in \mathbb{C}^{n+1}, | z_1|^2+\cdots+|z_n|^2 -|z_{n+1}|^2=-1 \}
\]
be the $2n+1$ dimensional anti-de Sitter space. We endow $\bH^{2n+1}$ with its standard Lorentz metric with signature $(2n,1)$. The Riemannian structure on $\mathbb{CH}^n$ is then such that the map
\begin{align*}
\begin{array}{llll}
\pi :& \bH^{2n+1} & \to & \mathbb{CH}^n \\
  & (z_1,\cdots,z_{n+1}) & \to & \left( \frac{z_1}{z_{n+1}}, \cdots, \frac{z_n}{z_{n+1}}\right)
\end{array}
\end{align*}
is an indefinite Riemannian submersion whose one-dimensional fibers are definite negative. This submersion is associated with a fibration. Indeed, the group $\mathbf{U}(1)$ acts isometrically on $\bH^{2n+1}$, and the quotient space of $\bH^{2n+1}$ by  this action is isometric to $\mathbb{CH}^n$.  The fibration
\[
\mathbf{U}(1)\to\bH^{2n+1}\to\mathbb{CH}^n
\]
 is called the anti-de Sitter fibration. 
 
 \
 
 To parametrize $\mathbb{CH}^n$, we will use the global inhomogeneous coordinates given by $w_j=z_j/z_{n+1}$ where $(z_1,\dots, z_n)\in \M$ with  $\M=\{z\in \mathbb{C}^{n,1}, \sum_{k=1}^n|z_{k}|^2-|z_{n+1}|^2<0 \}$. Let $\alpha$ be the one-form on $\mathbb{CH}^n$ which is defined in inhomogeneous coordinates by
\[
\alpha=\frac{i}{2(1-|w|^2)}\sum_{j=1}^n(w_jd\overline{w_j}-\overline{w_j}dw_j),
\]
where $|w|^2=\sum_{j=1}^n|w_j|^2<1$.  A simple computation yields
\[
d\alpha=\frac{i}{(1-|w|^2)^2}\left((1-|w|^2)\sum_{j=1}^ndw_j\wedge d\overline{w}_j-\sum_{j,k=1}^n\overline{w}_jw_k dw_j\wedge d\overline{w}_k \right).
\]
Thus $d\alpha$ is exactly the K\"ahler form which induces the standard Bergman metric on $\CH$. We can then naturally define the stochastic area process on $\CH$ as follows:
\begin{definition}
Let $(w(t))_{t \ge 0}$ be a Brownian motion on $\CH$ started at $0$ (We call $0$ the point with inhomogeneous coordinates $w_1=0,\cdots, w_{n}=0$). The generalized stochastic area process of $(w(t))_{t \ge 0}$ is defined by
\[
\theta(t)=\int_{w[0,t]} \alpha=\frac{i}{2}\sum_{j=1}^n \int_0^t \frac{w_j(s)  d\overline{w_j}(s)-\overline{w_j}(s) dw_j(s)}{1-|w(s)|^2},
\]
where the above stochastic integrals are understood in the Stratonovitch sense or equivalently It\^o sense.
\end{definition}

As in in the Heisenberg group case or the Hopf fibration case, the stochastic area process is intimately related to a diffusion on the total space of the fibration.

\begin{theorem}\label{FGHT}
 Let $(w(t))_{t \ge 0}$ be a Brownian motion on $\mathbb{CH}^n$ started at 0 and $(\theta(t))_{t\ge 0}$ be its stochastic area process. The $\bH^{2n+1}$-valued diffusion process
 \[
 Y_t=\frac{e^{i\theta_t} }{\sqrt{1-|w(t)|^2}} \left( w(t),1\right), \quad t \ge 0
 \]
 is the horizontal lift at $(0,1)$ of $(w(t))_{t \ge 0}$ by the submersion $\pi$.
\end{theorem}

\begin{proof}
Again, the key-point is to observe the compatibility of the submersion $\pi$ with the contact structure of $\bH^{2n+1}$. Namely, the horizontal distribution of this submersion is the kernel of the standard contact form on $\bH^{2n+1}$ 
\[
\eta=-\frac{i}{2}\left(\sum_{j=1}^{n}(\overline{z_j}dz_j-z_j d\overline{z_j})-(\overline{z_{n+1}}dz_{n+1}-z_{n+1} d\overline{z_{n+1}})\right).
\]

\

Let $(w_1,\cdots, w_n)$ be  the  coordinates for $\bH^{2n+1}$ given by $w_j=z_j/z_{n+1}$, and $\theta$ be the local fiber coordinate. These coordinates are given by the map
\begin{align*}
(w,\theta)\longrightarrow \frac{e^{i\theta} }{\sqrt{1-|w|^2}} \left( w,1\right),
\end{align*}
where  $\theta \in \R/2\pi\mathbb{Z}$, and $w \in \mathbb{CH}^n$. In these coordinates, we compute that
 \begin{align*}
\eta=-d\theta+\frac{i}{2(1-|w|^2)}\sum_{j=1}^n(w_jd\overline{w_j}-\overline{w_j}dw_j).
\end{align*}
As a consequence, the horizontal lift to $\bH^{2n+1}$  of the vector field $\frac{\partial}{\partial w_i}$ is given by $\frac{\partial}{\partial w_i}-\alpha\left(\frac{\partial}{\partial w_i} \right)\frac{\partial}{\partial \theta}$. Therefore, the lift of $(w(t))_{t \ge 0}$ is $\frac{e^{i\theta(t)} }{\sqrt{1-|w(t)|^2}} \left( w(t),1\right)$ with
\[
\theta(t) =\sum_{i=1}^n \int_0^t \alpha\left(\frac{\partial}{\partial w_i} \right)  dw_i =\int_{w[0,t]} \alpha.
\]
\end{proof}

Now we are in position to study the distribution of the process $(\theta_t)_{t \ge 0}$. First, we prove the following theorem:

\begin{theorem}\label{diff-H}
Let $r(t)=\arctanh |w(t)|$. The process $\left( r(t), \theta(t)\right)_{t \ge 0}$ is a diffusion with generator
 \[
L=\frac{1}{2} \left( \frac{\partial^2}{\partial r^2}+((2n-1)\coth r+\tanh r)\frac{\partial}{\partial r}+\tanh^2r\frac{\partial^2}{\partial \theta^2}\right).
 \]
 As a consequence the following equality in distribution holds
 \begin{equation}\label{eq-r-theta-H}
\left( r(t) ,\theta(t) \right)_{t \ge 0}=\left( r(t),B_{\int_0^t \tanh^2 r(s)ds} \right)_{t \ge 0},
\end{equation}
where $(B_t)_{t \ge 0}$ is a standard Brownian motion independent from $r$.
\end{theorem}

\begin{proof}
Similarly as in the case of $\mathbb{CP}^n$, we first compute the generator of the diffusion $Y$ introduced in the previous theorem. The Laplace-Beltrami operator for the Bergman metric of $\mathbb{CH}^n$ is given by
\[
\Delta_{\mathbb{CH}^n}=4(1-|w|^2)\sum_{k=1}^n \frac{\partial^2}{\partial w_k \partial\overline{w_k}}+ 4(1-|w|^2)\mathcal{R} \overline{\mathcal{R}}
\]
where
\[
\mathcal{R}=\sum_{j=1}^n w_j \frac{\partial}{\partial w_j}.
\] 
Thus, the generator of $Y$ is $\frac{1}{2}\bar{\Delta}_{\mathbb{CH}^n}$ where $\bar{\Delta}_{\mathbb{CH}^n}$ is the horizontal lift to  $\bH^{2n+1}$ of $\Delta_{\mathbb{CH}^n}$. We now observe that the horizontal lift of the vector field $\frac{\partial}{\partial w_i}$ to $\bH^{2n+1}$   is given by 
\[
\frac{\partial}{\partial w_i}-\alpha\left(\frac{\partial}{\partial w_i} \right)\frac{\partial}{\partial \theta}=\frac{\partial}{\partial w_j}+\frac{i}{2}\frac{\overline{w_j}}{1-\rho^2}\frac{\partial}{\partial\theta}.
\]
Substituting $\frac{\partial}{\partial w_i}$ by its lift in the expression of $\Delta_{\mathbb{CH}^n}$ yields
\[
\bar{\Delta}_{\mathbb{CH}^n}=4(1-|w|^2)\sum_{k=1}^n \frac{\partial^2}{\partial w_k \partial\overline{w_k}}- 4(1-|w|^2)\mathcal{R} \overline{\mathcal{R}}+|w|^2\ \frac{\partial^2}{\partial \theta^2}+2i(1-|w|^2)(\mathcal{R} -\overline{\mathcal{R}})\frac{\partial}{\partial\theta}.
\]
With $\rho=|w|$, we compute then that $\bar{\Delta}_{\mathbb{CH}^n}$ acts on functions depending only on $(\rho, \theta)$ as
\[
\left(1-\rho^2\right)^2\frac{\partial^2}{\partial \rho^2}+\left(\frac{(2n-1)(1-\rho^2)}{\rho}-(1-\rho^2)\rho\right)\frac{\partial}{\partial \rho}+\rho^2\frac{\partial^2 }{\partial \theta^2}.
\]
The change of variable $\rho =\tanh r$ finishes the proof.
\end{proof}

The heat kernel of $L$ has been computed and studied by Bonnefont for $n=1$ in \cite{Bon} and by J. Wang for $n\ge 2$ in \cite{JW}. As a consequence, we deduce:
\begin{itemize}
\item $\mathbf{n=1}$. For $t>0, r\ge 0, \theta \in \R$,
\[
\mathbb{P} \left( r(t) \in dr, \theta(t) \in d\theta\right)=\pi p_t(r,\theta)  \sinh 2r dr d\theta,
\]
where
\begin{align}\label{heatCH1}
p_t(r,\theta)=\frac{e^{-t/2}}{(2\pi t)^2} \int_{-\infty}^{+\infty} e^{-\frac{ \arcosh^2(\cosh r \cosh y)-(y-i\theta)^2}{2t } } \frac{\arcosh(\cosh r \cosh y) }{\sqrt{\cosh^2 r \cosh^2 y-1 } }dy.
\end{align}
\item $\mathbf{n \ge 2}$. For $t>0, r\ge 0, \theta \in \R$, 
\[
\mathbb{P} \left( r(t) \in dr, \theta(t) \in d\theta\right)=\frac{2\pi^n}{\Gamma(n)}p_t(r,\theta)(\sinh r)^{2n-1}\cosh r drd\theta,
\]
where
\begin{equation}\label{heatCHn}
p_t(r, \theta)=2\frac{\Gamma(n+1)e^{-n^2\frac{t}{2}+\frac{\pi^2}{2t}}}{(2\pi)^{n+2} t}\int_{-\infty}^{+\infty}\int_0^{+\infty}\frac{e^{\frac{(y-i\theta)^2-u^2}{2t}}\sinh u\sin\left(\frac{\pi u}{t}\right)}{\left(\cosh u+\cosh r\cosh y\right)^{n+1}}du
dy.
\end{equation}

\end{itemize}

These expressions of the heat kernel were derived using complex analysis and methods in partial differential equations.  In principle, the distribution of the couple $(r(t),\theta(t))$ is therefore completely known from an analytic point of view. However, these expressions are obviously not easy to work with. Actually when $n \ge 2$, it is not even apparent from the formula that $p_t(r,\theta) \ge 0$ ! In the next section, we propose a purely probabilistic methods that parallels the case of $\mathbb{CP}^n$ to compute the characteristic function of $\theta(t)$. The characteristic functions we obtain are easier to analyze.

\

Though, in general,  the formulas \eqref{heatCH1} and \eqref{heatCHn} are not easy to handle, in the case $n=1$, there is a very nice cancellation in the formula of $p_t(0,\theta)$ from which one can deduce the distribution of the stochastic area of the Brownian loop on $\mathbb{CH}^1$.

\begin{theorem}\label{HJKLP}
Assume $n=1$. For $ t>0$, and $\theta \in \R$,
\[
\mathbb{P} \left( \theta(t) \in d\theta | r(t)=0 \right)=\frac{1}{C(t)} \frac{e^{-\frac{\theta^2}{2t}}}{\cosh^2 \left( \frac{\pi \theta}{t}\right)} d\theta,
\]
where $C(t)$ is the normalization constant.
\end{theorem}

\begin{proof}
From Proposition 3.5 in \cite{Bon}, we have
\[
p_t(0,\theta)=\frac{e^{-t/2}}{2t^2}\frac{e^{-\frac{\theta^2}{2t}}}{\cosh^2 \left( \frac{\pi \theta}{t}\right)}.
\]
The result follows immediately.
\end{proof}

\subsection{Characteristic function of the stochastic area and limit theorem}

In this section we study the characteristic function of the stochastic area $\theta(t)$.
Let
\[
\mathcal{L}^{\alpha,\beta}=\frac{1}{2} \frac{\partial^2}{\partial r^2}+\left(\left(\alpha+\frac{1}{2}\right)\coth r+\left(\beta+\frac{1}{2}\right) \tanh r\right)\frac{\partial}{\partial r}, \quad \alpha,\beta >-1
\]
be the hyperbolic Jacobi generator. We will denote by $q_t^{\alpha,\beta}(r_0,r)$ the transition density with respect to the Lebesgue measure of the diffusion with generator $\mathcal{L}^{\alpha,\beta}$.

Let $\lambda \ge 0$, $r \in [0,+\infty)$ and 
\[
I(\lambda,r)=\mathbb{E}\left(e^{i \lambda \theta(t)}\mid r(t)=r\right).
\]
From Theorem \ref{diff-H},  we have
\begin{align*}
I(\lambda,r)& =\mathbb{E}\left(e^{i \lambda B_{\int_0^t \tanh^2 r(s)ds}}\mid r(t)=r\right) \\
 &=\mathbb{E}\left(e^{- \frac{\lambda^2}{2} \int_0^t \tanh^2 r(s)ds}\mid r(t)=r\right) 
\end{align*}
and $r$ is a diffusion with the  hyperbolic Jacobi generator
\[
\mathcal{L}^{n-1,0}=\frac{1}{2} \left( \frac{\partial^2}{\partial r^2}+((2n-1)\coth r+\tanh r)\frac{\partial}{\partial r}\right)
\]
started at $0$. 

\begin{theorem}\label{FTCHn}
For $\lambda \ge 0$, $r \in [0,+\infty)$, and $t >0$
\[
\mathbb{E}\left(e^{i \lambda \theta(t)}\mid r(t)=r\right)=\mathbb{E}\left(e^{- \frac{\lambda^2}{2} \int_0^t \tanh^2 r(s)ds}\mid r(t)=r\right) =\frac{e^{n \lambda t}}{(\cosh r)^\lambda} \frac{q_t^{n-1,\lambda}(0,r)}{q_t^{n-1,0}(0,r)}
\]
\end{theorem}

\begin{proof}
The proof follows the same lines as in Theorem \ref{FThj}. We have
\begin{equation}\label{eq-sde-r}
dr(t)= \frac{1}{2} \left( (2n-1)\coth r(t)+\tanh r(t) \right)dt+d\gamma(t),
\end{equation}
where $\gamma$ is a standard Brownian motion. For later use, we observe here that it implies that we almost surely have
\begin{align}\label{transient}
r(t) \ge \left( n -\frac{1}{2} \right) t +\gamma(t),
\end{align}
and thus $r(t)\to +\infty$ almost surely when $t \to \infty$.
Consider now the local martingale
\begin{align*}
D_t& =\exp \left( \lambda \int_0^t \tanh r(s) d\gamma(s) -\frac{\lambda^2}{2}  \int_0^t \tanh^2 r(s) ds \right)  \\
 &=\exp \left( \lambda \int_0^t \tanh r(s) dr(s)-\frac{\lambda}{2}(2n-1)t -\frac{\lambda +\lambda^2}{2}  \int_0^t \tanh^2 r(s) ds \right) 
\end{align*}
From It\^o's formula, we have
\begin{align*}
\ln \cosh r(t) & =\int_0^t \tanh r(s) dr(s)+\frac{1}{2} \int_0^t \frac{ds}{\cosh^2 r(s)} \\
 &=\int_0^t \tanh r(s) dr(s)-\frac{1}{2} \int_0^t \tanh^2 r(s) ds+\frac{1}{2} t.
\end{align*}
As a consequence, we deduce that
\[
D_t =e^{-n\lambda t} (\cosh r(t))^\lambda e^{- \frac{\lambda^2}{2} \int_0^t \tanh^2 r(s)ds}.
\]
We claim that $D_t$ is martingale. To see this we just need to show that $D_t$ is uniformly integrable, which will be implied by the fact that for all $1<p<+\infty$,
\[
\E\left(\sup_{s\le t} (D_t)^p\right)<+\infty.
\]
By Doob's maximal inequality, we just need to show $\E(D_t^p)<+\infty$ for all $1<p<+\infty$. This will follow from
$D_t\le (\cosh r(t))^\lambda $.  Indeed, from the comparison principle for stochastic differential equations, one has

\[
r(t) \le h(t),
\]
where $h$ solves the hyperbolic Bessel stochastic differential equation
\[
h(t)=n \int_0^t \coth h(s) ds +\gamma(t). 
\]
The distribution of the random variable $\cosh h(t)$ is well understood (see for instance Proposition 2.16 in \cite{Ja-Ma}) and has moments of any order. This implies that $D_t$ has moments of any order.

  Let us denote by $\mathcal{F}$ the natural filtration of $r$ and consider the probability measure $\mathbb{P}^\lambda$ defined by
\[
\mathbb{P}_{/ \mathcal{F}_t} ^\lambda=D_t \mathbb{P}_{/ \mathcal{F}_t}=e^{-n\lambda t} (\cosh r(t))^\lambda e^{- \frac{\lambda^2}{2} \int_0^t \tanh^2 r(s)ds} \mathbb{P}_{/ \mathcal{F}_t}.
\]
We have then for every bounded and Borel function $f$ on $[0,+\infty]$,
\begin{align*}
\mathbb{E}\left(f(r(t))e^{- \frac{\lambda^2}{2} \int_0^t \tanh^2 r(s)ds}\right)=e^{n\lambda t} \mathbb{E}^\lambda \left( \frac{f(r(t))}{(\cosh r(t))^\lambda} \right)
\end{align*}
From Girsanov theorem, the process
\[
\beta(t)=\gamma(t)-\lambda \int_0^t \tanh r(s) ds
\]
is a Brownian motion under the probability $\mathbb{P}^\lambda$. Since
\begin{equation}\label{eq-rt-H}
dr(t)= \frac{1}{2} \left( (2n-1)\coth r(t)+(2\lambda +1)\tanh r(t) \right)dt+d\beta(t),
\end{equation}
the proof is complete.
\end{proof}

As an immediate corollary of Theorem \ref{FTCHn}, we deduce an expression for  the characteristic function of the stochastic area process.
\begin{corollary}\label{FTh}
For $\lambda \in \mathbb{R}$ and $t \ge 0$,
\[
\mathbb{E}\left(e^{i \lambda \theta(t)}\right)=e^{n | \lambda | t}\int_0^{+\infty}  \frac{q_t^{n-1,| \lambda |}(0,r)}{(\cosh r)^{| \lambda|}} dr. 
\]
\end{corollary}

Let us point out that it appears difficult to directly invert the Fourier transforms in Theorem \ref{FTCHn} or Corollary \ref{FTh} to recover the formulas \eqref{heatCH1} and \eqref{heatCHn}.

\

We finish this section with a limit theorem for $\theta(t)$.
\begin{theorem}\label{LimitCHn}
When $t \to +\infty$, the following convergence in distribution takes place
\[
\frac{\theta(t)}{\sqrt{t}} \to \mathcal{N}(0,1)
\]
where $\mathcal{N}(0,1)$ is a normal distribution with mean 0 and variance 1.
\end{theorem}
\begin{proof}
We could use Corollary \ref{FTh}, but there is a shorter path that proves that the convergence is even almost sure.
From \eqref{transient}, one has $r(t)\to +\infty$ almost surely as $t\to+\infty$. It then follows that
\[
\coth r(t)\to 1,\quad \tanh r(t)\to 1\quad \mbox{a.s.}
\]
and hence
\[
\lim_{t\to+\infty}\frac{1}{t}\int_0^t \tanh^2 r(s)ds=1\quad \mbox{a.s.}
\]
Therefore, as a consequence of \eqref{eq-r-theta-H}, we have
\[
\lim_{t\to+\infty}\frac{\theta(t)}{\sqrt{t}}=\lim_{t\to+\infty}B_{\frac{1}{t}\int_0^t \tanh^2 r(s)ds}=B_1 \quad \mbox{a.s.}
\]
\end{proof}

\section{Stochastic winding numbers}

In this section, we study the winding number of the Brownian motion on $\mathbb{CP}^1$ and $\mathbb{CH}^1$. The winding number process is closely related to the stochastic area process, and the methods we developed allow us to prove  asymptotic results which are comparable to the classical Spitzer theorem for Brownian motion in the plane (see \cite{RY}, Theorem 4.1, Chapter X).  General results for the asymptotic windings of Brownian motions on compact Riemann surfaces are are well-known (see \cite{watanabe} and the references therein), and thus the Spitzer law in $\mathbb{CP}^1$ is already known. In the case of $\mathbb{CH}^1$ our result is new.

\subsection{Stochastic winding numbers in $\mathbb{CP}^1$}

As mentioned above, it is already known that Brownian motion windings on $\mathbb{CP}^1$ satisfy Spitzer's law.  However, it seems interesting to point out that Yor's  \cite{Yor1} methods apply to $\mathbb{CP}^1$. Our proof, will moreover yield an expression of the characteristic function of the winding process which was previously unknown.

The symmetric space $\mathbb{CP}^1$ is isometric to the two-dimensional Euclidean sphere with radius $\frac{1}{2}$. Therefore, the generator of the Brownian motion on  $\mathbb{CP}^1$ is in spherical coordinates
\[
\frac{1}{2} \left( \frac{\partial^2}{\partial r^2} + 2 \cot 2r \frac{\partial}{\partial r}  +\frac{4}{\sin^2 2r}  \frac{\partial^2}{\partial \phi^2}\right), \quad r\in[0,\pi/2], \phi \in [0,\pi],
\]
where $r$ parametrizes the Riemannian distance from the north pole on the sphere, which corresponds to 0 in $\mathbb{CP}^1$. This shows that  the winding number process of the Brownian motion on $\mathbb{CP}^1$ is given by
\[
\phi(t)=\beta_{\int_0^t \frac{4ds}{\sin^2 2r(s)}},
\]
where $r(t)$ is the Jacobi diffusion started at $r_0\in (0,\pi /2)$ with generator 
\[
\frac{1}{2} \left( \frac{\partial^2}{\partial r^2} + 2 \cot 2r \frac{\partial}{\partial r}  \right)
\]
and $\beta$ is a Brownian motion independent from $r$.

\begin{theorem}\label{windingCP}
When $t \to \infty$, in distribution we have
\[
\frac{\phi(t)}{t} \to \mathcal{C}_2,
\]
where $\mathcal{C}_2$ is a Cauchy distribution with parameter 2.
\end{theorem}

\begin{proof}
Let $\lambda >0$. We have 
\[
\mathbb{E}\left( e^{i \lambda \phi(t)} \right)=\mathbb{E}\left( e^{-\frac{\lambda^2}{2} \int_0^t \frac{4ds}{\sin^2 2r(s)}} \right)=e^{-2\lambda^2t} \mathbb{E}\left( e^{-2\lambda^2 \int_0^t \cot^2 2r(s) ds} \right).
\]
The process $r$ is solution of a stochastic differential equation
\[
r(t)=r_0+\int_0^t \cot 2r(s) ds +\gamma(t),
\]
where $\gamma$ is a Brownian motion. For $\lambda \ge 0$, let us consider the local martingale
\begin{align*}
D_t^\lambda =\exp \left( 2\lambda \int_0^t \cot 2r(s) d\gamma(s) -2\lambda^2 \int_0^t \cot^2 2r(s) ds \right)
\end{align*}
From It\^o's formula, we compute
\[
D_t^\lambda =e^{2\lambda t}(\sin 2r(t))^{\lambda} \exp\left(-2\lambda^2 \int_0^t \cot^2 2r(s) ds \right).
\]
Moreover, since $D_t\le e^{2\lambda t}$, we know that $D$ is a martingale. Now let us consider a new probability measure $\p^\lambda$ such that
\[
\mathbb{P}_{/ \mathcal{F}_t} ^\lambda=D_t \mathbb{P}_{/ \mathcal{F}_t}=(\sin 2r_0)^{-\lambda} e^{2\lambda t} (\sin 2r(t))^\lambda e^{- 2\lambda^2 \int_0^t \cot^2 2r(s)ds} \mathbb{P}_{/ \mathcal{F}_t},
\]
where $\mathcal{F}$ is the natural filtration of $r$. We have then
\[
\mathbb{E}\left( e^{i \lambda \phi(t)} \right)=(\sin 2r_0)^\lambda e^{-2(\lambda^2+\lambda) t}\E^{\lambda}\left((\sin2r(t)) ^{-\lambda}  \right) 
\]
By Girsanov's theorem we know that the process
\[
\beta (t)=\gamma(t)-2\lambda\int_0^t \cot 2r(s)ds 
\]
is a Brownian motion under $\p^\lambda$. As a consequence, since we have
\[
d r(t)= d\beta(t)+(2\lambda+1)\cot 2r(t)dt,
\] 
we deduce that $r$ is a Jacobi diffusion with generator
\[
\mathcal{L}^{\lambda,\lambda}=\frac{1}{2} \frac{\partial^2}{\partial r^2}+\left(\left(\lambda+\frac{1}{2}\right)\cot r-\left(\lambda+\frac{1}{2}\right) \tan r\right)\frac{\partial}{\partial r}
\]
under the probability $\mathbb{P}^\lambda$. Using now the fact that
\[
\mathbb{E}\left( e^{i \lambda \frac{\phi(t)}{t}} \right)=(\sin 2r_0)^{\frac{\lambda}{t}} e^{-2\left(\frac{\lambda^2}{t^2}+\frac{\lambda}{t}\right) t}\E^{\frac{\lambda}{t}}\left((\sin2r(t)) ^{-\frac{\lambda}{t}}  \right) 
\]
and the expression of $q_t^{\lambda,\lambda}(r_0,r)$ given in the Appendix, we see  that
\[
\lim_{t \to \infty} \mathbb{E}\left( e^{i \lambda \frac{\phi(t)}{t}} \right)=e^{-2\lambda}.
\]
\end{proof}

\subsection{Stochastic windings in $\mathbb{CH}^1$}

The space $\mathbb{CH}^1$ is isometric to the 
$2$-dimensional  hyperbolic space.
The generator of the Brownian motion on  $\mathbb{CH}^1$  in spherical coordinates is
\[
\frac{1}{2} \left( \frac{\partial^2}{\partial r^2} + 2 \coth 2r \frac{\partial}{\partial r}  +\frac{4}{\sinh^2 2r}  \frac{\partial^2}{\partial \phi^2}\right), \quad r\in[0,\infty), \phi \in [0,\pi],
\]
where $r$ parametrizes the Riemannian distance from $0$. This shows that  the winding number process of the Brownian motion on $\mathbb{CH}^1$ is given by
\[
\phi(t)=\beta_{\int_0^t \frac{4ds}{\sinh^2 2r(s)}},
\]
where $r(t)$ is the Jacobi diffusion started at $r_0\in (0,\infty)$ with generator 
\[
\frac{1}{2} \left( \frac{\partial^2}{\partial r^2} + 2 \coth 2r \frac{\partial}{\partial r}  \right)
\]
and $\beta$ is a Brownian motion independent from $r$.

\begin{theorem}\label{windingCH1}
When $t \to \infty$, in distribution we have
\[
{\phi(t)}\to \mathcal{C}_{\ln \coth r_0},
\]
where $\mathcal{C}_{\ln \coth r_0}$ is a Cauchy distribution with parameter $\ln \coth r_0$.
\end{theorem}
\begin{proof}
The process $r$ is solution of a stochastic differential equation
\[
r(t)=r_0+\int_0^t \coth 2r(s) ds +\gamma(t),
\]
where $\gamma$ is a standard Brownian Motion.  Let $\lambda>0$ and consider the local martingale
\[
D_t^\lambda=\exp \left( \lambda \int_0^t \frac{2d\gamma(s)}{\sinh 2r(s)} -\frac{\lambda^2}{2} \int_0^t \frac{4ds}{\sinh^2 2r(s)}\right).
\]
Using It\^o's formula, we find
\[
D_t^\lambda =\left(\frac{\tanh r(t)}{\tanh r_0}\right)^\lambda \exp \left( -\frac{\lambda^2}{2} \int_0^t \frac{4ds}{\sinh^2 2r(s)}\right).
\]
Therefore $D_t^\lambda $ is uniformly bounded, which implies that it is a martingale. Let us denote by $\mathcal{F}$ the natural filtration of $r$ and consider the probability measure $\mathbb{P}^\lambda$ defined by
\[
\mathbb{P}_{/ \mathcal{F}_t} ^\lambda=D^\lambda_t \mathbb{P}_{/ \mathcal{F}_t}.
\]
We have then
\begin{equation}\label{eq-E-tanh}
\mathbb{E} \left( e^{i \lambda \phi(t)} \right)=\mathbb{E}^\lambda \left( \left(\frac{\tanh r_0}{\tanh r(t)}\right)^\lambda\right). 
\end{equation}
Under the probability $\mathbb{P}^\lambda$ the process
\[
\beta(t)=\gamma(t)-\int_0^t \frac{2\lambda ds}{\sinh 2r(s)}
\]
is a Brownian motion. From \eqref{transient}, one has $r(t)\to +\infty$ almost surely as $t\to+\infty$ and therefore $\lim_{t \to \infty} \tanh r(t)=1$. By plugging back into \eqref{eq-E-tanh} we obtain
\[
\lim_{t \to \infty} \mathbb{E} \left( e^{i \lambda \phi(t)} \right)=(\tanh r_0)^\lambda.
\]
\end{proof}

\section{Appendix: Jacobi diffusions}\label{sec-appen}

We collect here some well-known facts about Jacobi diffusions (see for instance \cite{DZ} and the references therein for further details). The Jacobi diffusion is the diffusion on $[0,\pi/2]$ with generator
\[
\mathcal{L}^{\alpha,\beta}=\frac{1}{2} \frac{\partial^2}{\partial r^2}+\left(\left(\alpha+\frac{1}{2}\right)\cot r-\left(\beta+\frac{1}{2}\right) \tan r\right)\frac{\partial}{\partial r}, \quad \alpha,\beta >-1
\]
defined up to the first time it hits the boundary $\{ 0, \pi /2 \}$.

The point $0$ is:
\begin{itemize}
\item A regular point for $-1<\alpha <0$;
\item An entrance point for $\alpha \ge 0$.
\end{itemize}

Similarly, the point $\pi /2$ is:
\begin{itemize}
\item A regular point for $-1<\beta <0$;
\item An entrance point for $\beta \ge 0$.
\end{itemize}

If $r$ is a Jacobi diffusion with generator $\mathcal{L}^{\alpha,\beta}$, then it is easily seen that $\rho=\cos 2r$ is a diffusion with generator $2\mathcal{G}^{\alpha,\beta}$ where,

\begin{equation}\label{eq-G-ab}
\mathcal{G}^{\alpha,\beta}=(1-\rho^2)\frac{\partial^2}{\partial \rho^2}-\left( (\alpha+\beta+2)\rho +\alpha -\beta \right)\frac{\partial}{\partial \rho}
\end{equation}

The spectrum and eigenfunctions of $\mathcal{G}^{\alpha,\beta}$ are known. Let us denote by $P_m^{\alpha,\beta}(x)$, $m\in \mathbb{Z}_{\ge0}$ the  Jacobi polynomials given by
\[
P_m^{\alpha,\beta}(x)=\frac{(-1)^m}{2^mm!(1-x)^{\alpha}(1+x)^\beta}\frac{d^m}{dx^m}((1-x)^{\alpha+m}(1+x)^{\beta+m}).
\]
It is known that $\{P_m^{\alpha,\beta}(x)\}$ is orthonormal in $L^2([-1,1], 2^{-\alpha-\beta-1}(1+x)^{\beta}(1-x)^{\alpha}dx)$ and satisfies 
\[
\mathcal{G}^{\alpha,\beta}P_m^{\alpha,\beta}(x)=-m(m+\alpha+\beta+1)P_m^{\alpha,\beta}(x).
\]
If we denote by $p^{\alpha,\beta}_t(x,y)$ transition density with respect to the Lebesgue measure of the diffusion $\rho$ starting from $x \in (-1,1)$, then we have

\begin{align*}
 & p^{\alpha,\beta}_t(x,y)\\
 =&2^{-\alpha-\beta-1}(1+y)^{\beta}(1-y)^{\alpha}\sum_{m=0}^{+\infty}  (2m+\alpha+\beta+1)\frac{\Gamma(m+\alpha+\beta+1)\Gamma(m+1)}{\Gamma(m+\alpha+1)\Gamma(m+\beta+1)} e^{-2m(m+\alpha+\beta+1)t}P_m^{\alpha,\beta}(x)P_m^{\alpha,\beta}(y).
\end{align*}

In particular, when 1 is an entrance point, that is $\alpha \ge 0$,  we obtain

\begin{align*}
 & p^{\alpha,\beta}_t(1,y)\\
 =&2^{-\alpha-\beta-1}(1+y)^{\beta}(1-y)^{\alpha} \sum_{m=0}^{+\infty}  (2m+\alpha+\beta+1)\frac{\Gamma(m+\alpha+\beta+1)}{\Gamma(m+\beta+1)\Gamma(\alpha+1)} e^{-2m(m+\alpha+\beta+1)t}P_m^{\alpha,\beta}(y).
\end{align*}

By denoting $q_t^{\alpha,\beta}$ the transition density of $r$, we obtain then for $\alpha,\beta \ge 0$,

\begin{align*}
  q^{\alpha,\beta}_t(r_0,r)
 &=2(\cos r)^{2\beta +1} (\sin r)^{2\alpha +1}\\
 &\cdot
 \sum_{m=0}^{+\infty}  (2m+\alpha+\beta+1)\frac{\Gamma(m+\alpha+\beta+1)\Gamma(m+1)}{\Gamma(m+\alpha+1)\Gamma(m+\beta+1)}  e^{-2m(m+\alpha+\beta+1)t}P_m^{\alpha,\beta}(\cos 2r_0)P_m^{\alpha,\beta}(\cos 2r).
\end{align*}

and

\begin{align*}
 & q^{\alpha,\beta}_t(0,r)\\
 =&2(\cos r)^{2\beta +1} (\sin r)^{2\alpha +1} \sum_{m=0}^{+\infty}  (2m+\alpha+\beta+1)\frac{\Gamma(m+\alpha+\beta+1)}{\Gamma(m+\beta+1)\Gamma(\alpha+1)} e^{-2m(m+\alpha+\beta+1)t}P_m^{\alpha,\beta}(\cos 2r).
\end{align*}

\end{document}